\documentclass[leqno,11pt]{amsart}
\usepackage{amsfonts, amsmath, amssymb, lmodern, eurosym}
\usepackage[pdftex]{graphicx}
\usepackage[latin1]{inputenc}
\usepackage[english]{babel}
\usepackage[normalem]{ulem}
\usepackage{lettrine}
\usepackage{color}
\usepackage{lscape}
\usepackage[T1]{fontenc}
\usepackage[latin1]{inputenc}
\usepackage{tabularx}
\usepackage{array}
\usepackage{multirow}
\usepackage[Glenn]{fncychap}
\usepackage[english]{minitoc}
\usepackage[english]{algorithm2e}
\usepackage{listings}
\usepackage{verbatim}

\definecolor{blueumi}{rgb}{0.00,0.48,0.62}
\usepackage{tikz}
\usepackage{fancybox}
\usepackage{colortbl}
\usepackage{fancyhdr}
\usepackage{lastpage}

\newtheorem{definition}{Definition}[section]
\newtheorem{theo}{Theorem}[section]
\newtheorem{prop}[theo]{Proposition}
\newtheorem{Rq}{Remark}

\newtheorem{cor}{Corollary}

\usepackage{amsmath}
\usepackage{amssymb}
\usepackage{graphicx}
\usepackage{tikz-cd}
\usepackage[T1]{fontenc}
\usepackage{times}

\begin{document}
\pagestyle{plain}

\cfoot[C]{\thepage}


\vspace*{-1cm}





\title{On Topological Complexity of Gorenstein spaces}
\author[1]{Smail BENZAKI}
\author[2]{Youssef RAMI}

{\let\thefootnote\relax\footnote{{\it Address}: Moulay Ismaïl University, Department of Mathematics  B. P. 11 201 Zitoune,  Meknès, Morocco}\let\thefootnote\relax\footnote{{\it Emails}: smail.benzaki@edu.umi.ac.ma  and y.rami@umi.ac.ma}}



\renewcommand{\abstractname}{Abstract}
\begin{abstract}
In this paper, using Sullivan's approach to rational homotopy theory of simply-connected finite type CW complexes, we endow the  $\mathbb{Q}$-vector space $\mathcal{E}xt_{C^{\ast}(X;\mathbb{Q})}(\mathbb{Q},C^{\ast}(X;\mathbb{Q}))$ with a graded commutative algebra structure. This leads us to introduce the $\mathcal{E}xt$-version of higher  (resp. module, homology) topological complexity of $X_0$, the rationalization of $X$ (resp. of $X$ over $\mathbb{Q}$). We then make comparisons between these invariants and their respective ordinary  ones for Gorenstein spaces.
We also highlight, in this context, the benefit of Adams-Hilton models over a field of odd characteristics especially through two cases, the first one when the space is a  $2$-cell CW-complex and the second one when it is a suspension.
\end{abstract}

\maketitle

\section{Introduction}
Throughout this work  we assume that $X$ is a  $1$-connected CW-complex of finite type and $\mathbb{K}$  a field of  arbitrary characteristic. In his famous paper, M. Farber introduced the concept of topological complexity of a space $X$, $TC(X)$, which is  indeed seen as the {\it sectional category}, $secat(\Delta_X)$, of the diagonal map $\Delta_X : X\rightarrow X\times X$. Recall that if $f: X\rightarrow Y$ is a continuous map,  $secat(f)$ is the smallest integer $m$ for which there are $m+1$ local homotopy sections $s_i : U_i\rightarrow Y$ for $f$ whose sources form  an open covering of $X$. $TC(X)$ is actually  used to assess the complexity of motion planning algorithms of a mechanical system along its configuration space $X$. Later, Y. Rudyak in \cite{Ru} generalizes Farber's concept and introduces that of  higher topological complexity, $TC_n(X)$ ($n\geq 2$), which turns out to be the sectional category  $secat(\Delta^n_X)$ of the $n$-diagonal $\Delta^n_X : X\rightarrow X^n$. Determining $TC_n(X)$ is as difficult as it is exactly for the Lusternik-Scnirelmann category $cat_{LS}(X) = secta(*\hookrightarrow X)$ introduced in [LS]. Spectacular advances using rational homotopy theory methods  have however made it possible to establish better approximation of this latter, namely $cat_{LS}(X_0)\leq cat_{LS}(X)$ with $X_0$ denoting the rationalization $X_0$ of $X$. In that sense, J. Carrasquel used in \cite{JCar-V} the characterization \`a la F\'elix-Halperin to give an explicit definition  of  the higher topological complexity $TC_n(X_0)$ which in turn lowers $TC_n(X)$ (cf. Definition 4.1 below). This is achieved thanks to  the {\it Sullivan model} (called also {\it Sullivan algebra}) $(\Lambda V,d)$ associated to $X$ (cf. \S 2 for more details). He also introduced higher (rational) {\it   homology  (resp. module) topological complexity} $HTC_n(X)$, (resp. $MTC_n(X)$) and showed that they interpolate $zcl_n(X) := nil ker H^*(\Delta^n_X)$ and $TC_n(X_0)$.

 In this paper we are interested in the study of the invariant $HTC_n(X)$ when $X$ satisfies the {\it Gorenstein property} which we introduce right after.
Recall first that to every  local commutative ring $R$  with residue field  $\mathbb{K}$ it is associated the invariant $Ext_R(\mathbb{K}, R)$ and it is said  a {\it Gorenstein ring} if $\dim Ext_R(\mathbb{K}, R)=1$.

Motivated by the deep study of this invariant in the context of local algebra, the authors of \cite{FHT2} introduced the notion of a {\it Gorenstein space} and thereby obtain new characterizations of Poincaré duality complexes.
Recall from \cite{Sp} that  a finite $n$-dimensional subcomplex $X$  of $\mathbb{R}^{n+k}$ is a Poincar\'e complex if and only if its {\it Spivak fiber} $F_X$ (see \S 4) is an homotopy sphere.
They give, in  the particular case where $X$  has finite Lusternik-Schnirelmann category  $cat_{LS}(X)$  and $\mathbb{K}=\mathbb{Z}_p$ ($p$ prime or zero), a {\it local analogue of the Spivak characterization} in terms of the Eilenberg-Moore Ext functor denoted here by $\mathcal{E}xt$ \cite[Theorem 3.1]{FHT2}:
\begin{equation}\label{p-local}
   \begin{array}{lcl}
   H^*(X, \mathbb{Z}_p)\; \hbox{\it is a Poincar\'e duality algebra} & \; \Leftrightarrow & (F_X)_{(p)} \simeq  (\mathbb{S}^k)_{(p)}\\
   & \Leftrightarrow & \dim \mathcal{E}xt_{C^*(X;\mathbb{Z}_p)}(\mathbb{Z}_p, C^*(X;\mathbb{Z}_p)) =1
    \end{array}
\end{equation}
   (cf. \S 2 for the definition of Poincar\'e duality algebras and the functor $\mathcal{E}xt$). 
Recall that,
 $C^*(X;\mathbb{Z}_p)$ is endowed with the left module structure over itself.
The equivalences in (\ref{p-local}) make sense to the following definition \cite{FHT2}.
\begin{definition}\label{Gorenstein}
$X$ is a Gorenstein space over $\mathbb{K}$  if $\dim \mathcal{E}xt_{C^*(X;\mathbb{K})}(\mathbb{K}, C^*(X;\mathbb{K})) =1$.
\end{definition}
For instance, if $X$ is finite dimensional then \cite[Proposition 4.5, Theorem 3.1]{FHT2}:
$$
   \begin{array}{lcl}
   X \; \hbox{is Gorenstein over} \; \mathbb{K} &  \Leftrightarrow & H^*(X, \mathbb{K})\; \hbox{ is a Gorenstein  algebra}\\
   & \Leftrightarrow & H^*(X, \mathbb{K}) \; \hbox{is a Poincaré duality algebra}.
    \end{array}
$$
 E.g. every closed manifold is Gorenstein over any field $\mathbb{K}$.
 In the rational case, if $\dim \pi_*(X)\otimes \mathbb{Q}$ is finite dimensional then  $X$ is Gorenstein over $\mathbb{Q}$ \cite[Proposition 3.4]{FHT2}. E.g. every rationally elliptic space  $X$ (i.e.   both if its rational homology and rational homotopy are finite dimensional) is Gorenstein over $\mathbb{Q}$.

The authors of \cite{FHT2} also introduced a morphism of $\mathbb{K}$-graded vector spaces:
$$ev_{C^*(X, \mathbb{K})}: \mathcal{E}xt_{C^*(X, \mathbb{K})}(\mathbb{K},C^*(X, \mathbb{K})\rightarrow H^*(X, \mathbb{K})$$
(cf. \S 2)   called the {\it evaluation map} of $X$ over $\mathbb{K}$. This is a homotopy invariant which  allows to better (algebraically) characterize  Poincar\'e duality spaces \cite[Corollary 2]{Ga}:
$$
X \;  \hbox{is a Poincaré duality space}   \Leftrightarrow \\
\left \{ \begin{array}{l}
(i)\; X \; \hbox{is a Gorenstein space},\\
(ii)\; ev_{C^*(X, \mathbb{K})}\not =0\\
(iii)\;  H^*(X, \mathbb{Q}) \; \hbox{is a Gorenstein algebra}
\end{array}
\right.
$$
Also, using the evaluation map we know, for instance, that every space  $Y= X\cup e^n$ with $e^n$ representing a non-zero class in $H_n(X, \mathbb{Q})$ is Gorenstein over $\mathbb{Q}$ \cite[Corollary 2.3 and Proposition 3.4]{FHT2}.

In fact, fewer space types for which the  invariants $HTC_n(X)$, $mTC_n(X)$ and $TC_n(X_0)$ are determined. However,  in \cite{HRV}   a new invariant $L(X_0)$ is introduced  and it is shown that  $cat_{LS}(X_0)+L(X_0)\leq TC_n(X_0)$ for any pure  elliptic coformal space. It is also established that $TC_n(X_0)=\dim \pi_*(X)\otimes \mathbb {Q}$ for certain particular families of such spaces.

In this work, inspired by Carrasquel's characterizations we will introduce $Ext$-versions of the aforementioned topological complexities.  We start by endowing the  graded $\mathbb{Q}$-vector space
$$\mathrm{A}=:Hom_{(\Lambda V, d)}(\Lambda V\otimes \Lambda(sV),(\Lambda V, d))$$ (cf. \S 2 for more details) with a homotopy multiplicative structure denoted in all that follows
\begin{equation}
\mu_{\mathrm{A}}: \mathrm{A}\otimes \mathrm{A}\rightarrow \mathrm{A}.
\end{equation}
The cohomology of $\mathrm{A}$ is neither than $\mathcal{A} =\mathcal{E}xt_{(\Lambda V, d)}(\mathbb{Q},(\Lambda V, d))$ and  our first main result reads:
\begin{theo}\label{Th1}
The graded vector space $\mathcal{A}$ endowed with the product $\mu_{{\mathcal{A}}}:=H^*(\mu_{\mathrm{A}})$ is a graded commutative algebra with unit. Moreover, the evaluation map $$ev_{(\Lambda V,d)}:\mathcal{E}xt_{(\Lambda V,d)}(\mathbb{Q},(\Lambda V,d))\rightarrow H(\Lambda V,d)$$ is a morphism of graded algebras.
\end{theo}
As a consequence, using \cite[Remark 1.3]{FHT2} we have the following:
\begin{cor}\label{Cor1}
The graded vector space $\mathcal{E}xt_{C^*(X, \mathbb{Q})}(\mathbb{Q},C^*(X, \mathbb{Q}))$ is a graded commutative algebra and the evaluation map $$ev_{C^*(X, \mathbb{Q})}:\mathcal{E}xt_{C^*(X, \mathbb{Q})}(\mathbb{Q},C^*(X, \mathbb{Q}))\rightarrow H^*(X, \mathbb{Q})$$ is a morphism of graded algebras.
\end{cor}
Next, recall from \cite{JCar}  that a  morphism of graded commutative differential algebra (cdga for short)  $\varpi: (C, d)\rightarrow (D, d) $ admits  a {\it homotopy retraction} if there exists a map $r: (\Lambda V\otimes C, d))\rightarrow (C, d)$ such that $r\circ \iota =Id_C$, where  $\iota: (C, d)\rightarrow (\Lambda V\otimes C, d)$ stands for the relative model of $\varpi$.
As a particular case, consider any Sullivan algebra $(\Lambda V,d)$ and denote by $\mu_{\mathrm{A},n} : \mathrm{A}^{\otimes n}\rightarrow \mathrm{A}$ the $n$-fold product of $\mu_{\mathrm{A}}: \mathrm{A}\otimes \mathrm{A}\rightarrow \mathrm{A}$.
Using  \cite[Definition 9]{JCar-V}, we state in terms of the cdga projection:
$$\Gamma_{m}:\left(\mathrm{A}^{\otimes n}, d\right) \rightarrow\left(\frac{\mathrm{A}^{\otimes n}}{\left(\operatorname{ker} {{(\mu_{\mathrm{A},n})}}\right)^{m+1}}, \overline{d} \right)$$
the following
\begin{definition}\label{sc-Ext}
\begin{enumerate}
{\it\item[(a):] $sc(\mu_{\mathrm{A},n})$ is the least $m$ such that $\Gamma_m$
admits a homotopy retractioe  
\item[(b):] $msc(\mu_{\mathrm{A},n})$ is the least $m$ such that $\Gamma_{m}$ admits a homotopy retraction as $(\mathrm{A}^{\otimes n}, d)$-modules.
\item[(c):] $Hsc(\mu_{\mathrm{A},n})$ is the least $m$ such that $H\left(\Gamma_{m}\right)$ is injective.
\item[(d):] $nil \ker{(\mu_{\mathcal{A},n}, \mathbb{Q})}$ is the longest non trivial product of elements of $\ker {(\mu_{\mathcal{A},n})}$.}

\hspace*{-1.7cm}  If $(\Lambda V,d)$ is a Sullivan model of $X$, these are denoted respectively  $\mathrm{TC}^{\mathcal{E}xt}_{n}\left(X, \mathbb{Q}\right)$, $\mathrm{mTC}^{\mathcal{E}xt}_{n}\left(X, \mathbb{Q}\right)$
\hspace*{-1.3cm}   $\mathrm{HTC}^{\mathcal{E}xt}_{n}\left(X, \mathbb{Q}\right)$ and $zcl_n^{\mathcal{E}xt}(X, \mathbb{Q})$.
 We call theme the rational (resp. module, homology, zero cup length) \\
 \hspace*{-1.3cm} Ext-version topological complexity.
\end{enumerate}
\end{definition}

Our second main result reads as follows
\begin{theo}\label{main1}
Let $X$ be  a $1$-connected finite type CW-complex.
 If $X$ is  a Gorenstein space over $\mathbb{Q}$ and $ev_{C^*(X, \mathbb{Q})}\not =0$,
then $HTC^{\mathcal{E}xt}_n(X, \mathbb{Q})\leq HTC_n(X, \mathbb{Q})
$  for any integer $n\geq 2$.
Furthermore, if $(\Lambda V, d)$ is a  Sullivan minimal model  of $X$, $[f]$ is the generating class of $\mathcal{A}$ and $m=HTC^{\mathcal{E}xt}_n(X, \mathbb{Q})$, then:
 $$HTC_n(X, \mathbb{Q}) = HTC^{\mathcal{E}xt}_n(X, \mathbb{Q})=:m\Leftrightarrow 
 f(1)^{\otimes n}\in (\ker{{\mu_{n}}})^{m}\backslash (\ker{{\mu_{n}}})^{m+1}.$$
\end{theo}
Notice, referring to \cite{Ga} that hypothesis on $X$ imply that it is either a Poincar\'e duality space on $\mathbb{Q}$ or else $H^*(X, \mathbb{Q})$ is not noetherian and not a Gorenstein graded algebra.
It also should be pointed out here that, when $X$ satisfies Poincar\'e duality property then $HTC_n(X, \mathbb{Q}) = mTC_n(X, \mathbb{Q}) $ \cite[Corollary 4.8]{JCar}.
As Particular cases, we have:
\begin{cor}\label{Cor2}
Hypothesis of Theorem \ref{main1} and hence its conclusions are satisfied in the following cases:
\begin{enumerate}
\item[(a)]  $X$ is rationally elliptic,
\item[(b)]   $H_{>N}(X,\mathbb{Z})=0$, for some $N$,  and $H^*(X, \mathbb{Q})$ is a Poincaré duality algebra,
\item[(c)] $X$ is a  finite $1$-connected CW-complex and its Spivak fiber $F_X$ has finite dimensional cohomology.
\end{enumerate}
\end{cor}

Next, let $R$ be a principal ideal domain containing $\frac{1}{2}$. Denote by $\rho(R)$ the least  non invertible   prime ( or $\infty$) in $R$ and  by 
$CW_r(R)$ the sub-category of finite $r$-connected CW-complexes $X$ ($r \geq 1$) satisfying  $\dim(X)\leq r\rho(R)$.


In \cite{Ha} (see also \cite{A}),  in his attempt to extend Sullivan's theory to arbitrary rings, S. Halperin associated to every $X\in CW_r(R)$  an appropriate differential graded Lie algebra  $(L, \partial)$  and showed that its Cartan-Eilenberg-Chevally complex $C^*(L, \partial)$ is, on one hand, linked with the cochains algebra  $C^*(X, R)$  by a series of quasi-isomorphisms (\cite[p. 274]{Ha}) and, on the other hand, it is  quasi-isomorphic to
a free commutative differential graded algebra $(\Lambda W, d)$  \cite[\S 7]{Ha}. $(\Lambda W, d)$ is then called {\it a free commutative model} of $X$ or a {\it Sullivan minimal model} of $X$. 

In Section 5 we will extend a part of Theorem \ref{Th1} and its corollary to the sub-category $CW_r(R)$ when $R=\mathbb{K}$ is a field of odd characteristic and make use of  Adams-Hilton models introduced in \cite{AH} to obtain an explicit calculation of the homotopy invariant $\mathcal{A}=\mathcal{E}xt_{C^*(X, \mathbb{K})}(\mathbb{K},C^*(X, \mathbb{K}))$ when $X$ is a suspension or a two-cell CW-complex.
 
The rest of the paper is organized as follows: In Section 2 we recall notions used to state and to establish our results while in Section 3 we introduce   the  multiplicative structure $\mu_{\mathcal{A}}$ and the proof of Theorem \ref{Th1}. Section 4 is devoted to the proof of Theorem \ref{main1}.



\section{Preliminaries}
To make this work constructive enough, we summarize in this section most of the results established on Gorenstein algebras and Gorenstein spaces. $\mathbb{K}$ will denote an arbitrary ground field  unless otherwise stated.

\subsection{Eilenberg-Moore Ext:}




A graded module  is a family $A= (A^i)_{i\in \mathbb{Z}}$ of $\mathbb{K}$-modules denoted also $A = \oplus _{i\in \mathbb{Z}}A^i$. Every $a\in A^i$ is of degree $i$  denoted thereafter $|a|$.

A linear map of graded modules $f: A\rightarrow B$ of degree $|f|$ is a $\mathbb{K}$-linear map sending each $A^i$ to $B^{i+|f|}$. In $|f|=0$ we call it a morphism of graded modules.

{\it In all that follows, unless otherwise stated, modules  are over $\mathbb{K}$ and we will assume that $A^i=0$ if $i<0$}.

A grade algebra $A$ is a graded module together with an associative multiplication $\mu_A : A\otimes A\rightarrow A$ that has un identity element $1_A=: 1\in A^0$. We will put $\mu_A(x\otimes y)=: xy$. Notice that $|\mu_A|=0$. Moreover, if we have $ab=(-1)^{|a||b|}ba$ for all $a,b\in A$, then $A$ is said to be commutative. 

A differential graded algebra $(A, d)$ (dga for short) is a graded algebra $A$ together with a linear map $d: A\rightarrow A$ of degree $|d|=+1$ that is a derivation : $d(ab)=d(a)b + (-1)^{|a|}ad(b)$,  and satisfying $d\circ d=0$.

A morphism of dga $f : (A,d)\rightarrow (B,d)$ is a linear map of degree zero satisfying 
$f(aa') = f(a)f(a')$, and the compatibility with the differential $d$: $f(da) = d(f(a))$.

A dga algebra $A$ is said to be augmented if it is endowed with  a morphism $\varepsilon : A\rightarrow \mathbb{K}$ of graded algebras.

A (left) graded $(A, d)$ module  is a graded module $M$ equipped  with a linear map    $A\otimes M\rightarrow M$, $a\otimes m\mapsto am$ of degree zero such that $a(bm)=(ab)m$ and $1m=m$, and a differential $d$ satisfying
$d(am)=(da)m+(-1)^{|a|}a(dm), \hspace*{0.2cm}m\in M,\hspace*{0.1cm}a\in A.$

A morphism of  (left) graded modules over a dga $(A,d)$ is a morphism $f:(M,d)\rightarrow (N,d)$ compatible with the differential: $d\circ f = f\circ d$.

 A left $(A, d)$-module $(M, d)$ is said semi-free if it is the union of an increasing sequence $M(0)\subset M(1)\subset M(2)\cdots\subset M(n)\subset \cdots$ of sub $(A,d)$-modules such that $M(0)$ and each $M(i)/M(i-1)$ is $A$-free on a basis of cycles. Such an increasing sequence is called a semi-free filtration of $(M,d)$.

A semi-free resolution of an $(A, d)$-module $(M, d)$ is an $(A, d)$-semi-free module $(P, d)$ together with a quasi-isomorphism (i.e. a morphism inducing an isomorphism in homology) $\begin{tikzcd}[column sep=small] m:(P,d)\arrow[r,"\simeq"] & (M,d) \end{tikzcd}$ of $(A,d)$-modules. Each of  $M(0)$ and  $M(i)/M(i-1)$ has the form $(A,d)\otimes (V(i),0)$ where $V(i)$ is a free $\mathbb{K}$-module. Thus the surjections  $M(n)\rightarrow A\otimes V(n)$ and the differential $d$ satisfy :
$$\begin{array}{ccc}
M(n)=M(n-1)\oplus (A \otimes V(n)), & and & d:V(n)\rightarrow M(n-1).
\end{array} $$
By \cite[Prop. 6.6]{FHT}, every $(A,d)$-module $(M,d)$ has a semi-free resolution $\begin{tikzcd}[column sep=small] m:(P,d)\arrow[r,"\simeq"] & (M,d) \end{tikzcd}$ and if $\begin{tikzcd}[column sep=small] m:(P',d)\arrow[r,"\simeq"] & (M,d) \end{tikzcd}$ is a second semi-free resolution, then, there is an equivalence  $$\begin{tikzcd} \alpha:(P',d)\arrow[r] & (P,d) \end{tikzcd}$$ of $(A,d)$-modules such that $m\circ \alpha$ and  $m'$ are homotopic morphisms; we denote for short $m\circ \alpha \simeq_A m'$.

Particularly, let $(A,d)$ be a differential graded algebra and $\begin{tikzcd}[column sep=small] (P,d)\arrow[r,"\simeq"] & (\mathbb{Q},0)   \end{tikzcd}$  an $(A,d)$-semi-free resolution of $(\mathbb{Q},0)$. This defines the graded $(A,d)$-module
\begin{align*}
Hom_A((P,d),(A,d))=\bigoplus_{p\geq 0}Hom_A^{p,\ast}((P,d),(A,d))
 =\bigoplus_{p\geq 0}   \bigoplus_{i\geq 0}Hom_A(P^i,A^{i+p}),
\end{align*}
which, endowed with the differential
$$D(f)=d\circ f-(-1)^pf\circ d; \quad f\in Hom_A^{p,\ast}((P,d),(A,d)),$$ yields the Eilenberg-Moore Ext functor:
$$\mathcal{E}xt_{(A,d)}(\mathbb{K},(A,d))=H^{\ast}(Hom_A((P,d),(A,d)),D).$$
This  is an invariant up to  homotopy of differential graded algebras (see \cite[Appendix]{FHT2} or \cite[Appendix]{FH1}). More usefully  \cite[Remark 1.3]{FHT2} if
 $(A, d)\stackrel{\simeq}{\longrightarrow} (B, d)$ is a quasi-isomorphism of differential graded algebras, then $\mathcal{E}xt_{(A,d)}(\mathbb{K},(A,d))$ is identified with $\mathcal{E}xt_{(B,d)}(\mathbb{K},(B,d))$  via natural (induced) isomorphisms
\begin{equation}\label{natiso}
\mathcal{E}xt_{(A,d)}(\mathbb{K},(A,d))\stackrel{\cong}{\longrightarrow} \mathcal{E}xt_{(A,d)}(\mathbb{K},(B,d))\stackrel{\cong}{\longleftarrow} \mathcal{E}xt_{(B,d)}(\mathbb{K},(B,d)).
\end{equation}
Particularly, $\mathcal{E}xt_{C^{\ast}(X;\mathbb{K})}(\mathbb{K},C^{\ast}(X;\mathbb{K}))$ and  $\mathcal{E}xt_{C_{\ast}(\Omega X;\mathbb{K})}(\mathbb{K},C_{\ast}(\Omega X;\mathbb{K}))$  depend only on the homotopy class of  $X$.

The highest $N$ such that  $[\mathcal{E}xt_{C^{\ast}(X;\mathbb{K})}(\mathbb{K},C^{\ast}(X;\mathbb{K}))]^N\not =0$ is called the formal dimension of $X$. It is denoted $fd(X, \mathbb{K})$.
\subsection{Evaluation map and Gorenstein spaces.}

Let $\rho:(P, d) \stackrel{\simeq}{\rightarrow}(\mathbb{K}, 0)$ be a minimal $(A,d)$-semi-free resolution of $(\mathbb{K},0)$. Consider the chain map
$$cev_{(A,d)}: \operatorname{Hom}_{(A, d)}((P, d),(A, d)) \longrightarrow(A, d)$$
given by $f \mapsto f(z)$, where $z \in P$ is a cocycle representing $1$ in $\mathbb{K}$. We call it the {\it chain evaluation map of $(A,d)$}. Passing to homology, we obtain the natural map
$$ ev_{(A, d)}: \mathcal{E} x t_{(A, d)}(\mathbb{K},(A, d)) \longrightarrow H^{*}(A, d), $$ called the {\it evaluation map of $(A,d)$}. The definition of $ev_{(A,d)}$ is independent of the choice of $(P, d)$ and $z$. The evaluation map of $X$ over $\mathbb{K}$ is by definition the {\it evaluation map} of $C^{*}(X, \mathbb{K})$.

A {\it Poincaré duality algebra} over $\mathbb{K}$  is a graded
algebra $H = \{H^k\}_{0\leq k\leq N}$ such that $H^N = \mathbb{K}\alpha$ and the pairing
$< \beta,\gamma >\alpha = \beta \gamma, \; \beta \in H^k,\; \gamma \in H^{N-k}$
defines an isomorphism $H^k \stackrel{\cong}{\rightarrow}Hom_{\mathbb{K}}(H^{N-k}, \mathbb{K})$, $0 \leq  k\leq  N$. In particular, $H = Hom_{\mathbb{K}}(Hom_{\mathbb{K}}(H, \mathbb{K}), \mathbb{K})$ is necessarily finite dimensional.

 A {\it Poincaré space}  at $\mathbb{K}$ is  a space   whose cohomology with coefficients in $\mathbb{K}$ is a Poincaré duality algebra. In this case, the cohomology class $\alpha$ such that $H^N(X, \mathbb{K}) = \mathbb{K}\alpha$ has degree $N=fd(X, \mathbb{K})$\cite[Proposition 5.1]{FHT2}. It is called 
 the  {\it fundamental class} $X$.

A {\it Gorenstein algebra} over $\mathbb{K}$ is a differential graded algebra $(A, d)$ whose associated  graded vector space $\mathcal{E}xt_{(A,d)}(\mathbb{K},(A,d))$ is one dimensional. 

 A space $X$ is  {\it Gorenstein} over $\mathbb{K}$ if the cochain algebra $C^*(X;\mathbb{K})$ is a Gorenstein algebra. 

For instance, let  $X$ be a  simply connected CW complex. If in addition $X$ it finite dimensional then: $C^*(X;\mathbb{K})$ is Gorenstein if and only if $H^*(X;\mathbb{K})$ is a Poinca\'e duality algebra \cite[Theorem 3.1]{FHT2}. But, if $X$ is not  finite dimensional,   it is  Gorenstein  provided that
$\dim \pi_*(X) \otimes \mathbb{Q} < \infty$ \cite[Proposition 3.4]{FHT2}. In  this case,   $\dim H^*( X, \mathbb{Q}) < \infty$ if and only if $ev_{C^*(X;\mathbb{Q})}\not = 0$ \cite{Mu}.


\section{The $\mathbb{Q}$-algebra $\mathcal{E}xt_{(\Lambda V,d)}(\mathbb{Q},(\Lambda V,d))$}
Along this section, the ground field is $\mathbb{Q}$. Recall from \cite{FHT} that to every finite-type simply-connected space $X$   it is associated a quasi-isomorphism $(\Lambda V,d)\stackrel{\simeq}{\rightarrow} A_{PL}(X)$ from a free commutative differential graded algebra (cdga for short) $(\Lambda V,d)$ to the commutative graded algebra $A_{PL}(X)$ of polynomial forms with rational coefficients. This latter    is connected to  $C^*(X, \mathbb{Q})$ by a sequence of quasi-isomorphisms. More explicitly, $\Lambda V= TV/I$ where $I$ is the graded ideal spanned by $\{v\otimes w - (-1)^{deg(u)deg(v)}w\otimes v, \; v,\; w\in V\}$,  $V=\oplus_{n\geq 2}V^n$ is a
finite-type graded vector space
 and the differential $d$ is a derivation defined on $V$ satisfying $d\circ d=0$. $(\Lambda V,d)$ is called a {\it Sullivan model} of $X$. This model is said {\it minimal} if  moreover
$d$   is decomposable, i.e. $d: V\rightarrow \Lambda^{\geq 2} V$ where $\Lambda^{\geq 2} V$ denotes the graded vector space spanned by all the monomials $v_1\ldots v_r$
 ($r\geq 2$). Notice that such a model is unique up to isomorphisms \cite{FHT}.

Let $(X, x_0)$ be a based simply-connected finite type CW-complex and denote by
$$\begin{tikzcd} m: (\Lambda V,d)\arrow[r,"\simeq"] & A_{PL}(X) \end{tikzcd}$$
 its minimal Sullivan model. In fact,    the multiplicative structure of $(\Lambda V,d)$,  
 $\mu_{\Lambda V} : \Lambda V\otimes_{\mathbb{Q}} \Lambda V\rightarrow \Lambda V$ is compatible with the one  induced on $C^*(X, \mathbb{Q})$ by the diagonal map $\Delta_X : X\rightarrow X\times X$ and the same holds for the augmentation $\varepsilon _{\Lambda V} : (\Lambda V, d)\rightarrow (\mathbb{Q}, 0)$ and the inclusion $\iota : \{x_0\}\hookrightarrow X$. 
 
 Referring for instance to \cite{FHJLT} a $(\Lambda V,d)$-semi-free resolution of $(\mathbb{Q},0)$ has the form $(P, d)=(\Lambda V\otimes \Lambda sV, d)\stackrel{\simeq}{\longrightarrow} (\mathbb{Q}, 0)$ where $sV$ is the suspension of $V$ defined by $(sV)^k=V^{k+1}$ and $d(sv) = - s(dv)$ for all $v\in V$.

 We are mow ready to  define on $\mathrm{A}=:Hom_{(\Lambda V, d)}(\mathbb{Q},(\Lambda V, d))$ an homotopy
  multiplication:
  \begin{equation}
  \mu_{\mathrm{A}}: \mathrm{A}\otimes \mathrm{A}\rightarrow \mathrm{A}
  \end{equation}
  which induces in cohomology a multiplicative structure:
  \begin{equation}
    \mu_{\mathcal{A}}: \mathcal{A}\otimes \mathcal{A}\rightarrow \mathcal{A}
    \end{equation}
on   $\mathcal{A} =: \mathcal{E}xt_{(\Lambda V,d)}(\mathbb{Q},(\Lambda V,d))$. The proof of Theorem 1.1 and Corollary 1 follows from the following theorem.
\begin{theo}
 The $\mathbb{Q}$-vector space $\mathcal{A}$, endowed with $\mu_{{\mathcal{A}}}$, is a graded commutative algebra with unit. Moreover, the evaluation map is a morphism of graded algebras.
\end{theo}
\begin{proof}
Let $f,g:P\rightarrow \Lambda V$ be elements in $\mathrm{A}$ representing  two classes in $\mathcal{A}$.
As $\Lambda V$ is commutative, the left $\Lambda V$-module $P$ is also a  right $ \Lambda V$-module  by setting $x\cdot a=(-1)^{|x||a|}a\cdot x$, \hspace*{0.2cm} $x\in P$ and $a\in \Lambda V$.
\\\\
\textit{\textbf{Multiplicative structure:}}

First, we consider
\begin{align*}
f\otimes g: P\otimes_{\mathbb{Q}} P    & \longrightarrow  \Lambda V\otimes_{\mathbb{Q}} \Lambda V\\
            x\otimes y \hspace*{0.3cm} & \longmapsto (-1)^{|g||x|}f(x)\otimes g(y),
\end{align*}
and $I$ the ideal generated by $x\cdot a\otimes y-x\otimes a\cdot y$; $x,y\in P$ and $a\in  \Lambda V$. It is straightforward that the map $\mu_{\Lambda V}\circ (f\otimes g)$ sends $I$ to zero, which then induces on the quotient $P\otimes_{\Lambda V} P= P\otimes _{\mathbb{Q}}P/I$, the dashed  map
$$\begin{tikzcd}
P\otimes_{\mathbb{Q}} P \arrow[r,"f\otimes g"]\arrow[d] &  \Lambda V\otimes_{\mathbb{Q}} \Lambda V \arrow[r,"\mu_{ \Lambda V}"] &  \Lambda V\\
P\otimes_{ \Lambda V} P \arrow[rru,dashed,"\mu_{\mathrm{A}}(f\otimes g)"'] & &
\end{tikzcd}$$
where
\begin{align*}
\mu_{\mathrm{A}}(f\otimes g)=: f.g : P\otimes_{ \Lambda V} P    & \longrightarrow  \Lambda V\\
            x\otimes y \hspace*{0.3cm} & \longmapsto (-1)^{|g||x|}f(x)g(y),
\end{align*}

Now, for $a\in \Lambda V$ and $x,y\in P$, we have $(f\cdot g)((a\cdot x)\otimes y)=(-1)^{|f\cdot g||a|}a(f\cdot g)(x\otimes y)$. Therefore $f\cdot g$ is an $ \Lambda V$-morphism.

Next, we show that $Q=P\otimes_{ \Lambda V}P$ is an $\Lambda V$-semi-free resolution. Recall that
a semi-free resolution $(P,d)$ of $\mathbb{Q}$ has the form $W\otimes_{\mathbb{Q}}\Lambda V$ with $W=\bigoplus_{i=0}^{+\infty}W(i)$ and each $W(i)$ is a free graded $\mathbb{Q}$-module and $d:W(k)\rightarrow P(k-1)$, with the semi-free filtration  given by $P(k)=\bigoplus_{i=0}^k W(i)\otimes_{\mathbb{Q}}\Lambda V$ \cite{FHT}.
Therefore, $$Q=(W\otimes_{\mathbb{Q}} \Lambda V)\otimes_{ \Lambda V}(W\otimes_{\mathbb{Q}} \Lambda V)=(W\otimes_{\mathbb{Q}}W)\otimes_{\mathbb{Q}} \Lambda V.$$
Let   $Z=W\otimes_{\mathbb{Q}} W$  and  put
$Q(k)=\bigotimes_{i=0}^k Z(i)\otimes_{\mathbb{Q}} \Lambda V$ where $Z(l)=\bigoplus_{i+j=l}W(i)\otimes_{\mathbb{Q}} W(j)$  is obviously a free graded $\mathbb{Q}$-module since each $W(i)$ is.
For any $x\otimes y\in W(i)\otimes W(j)$, we verify easily that
$$dx\otimes y\in P(i-1)\otimes W(j)\subseteq Q(k-1) \; \hbox{and}\;
x\otimes dy\in W(i) \otimes P(j-1)\subseteq Q(k-1),$$
 whence $D:Z(k)\rightarrow Q(k-1)$.
It results that $\begin{tikzcd} (Q,D)\arrow[r,"\simeq"] & (\mathbb{Q},0)\end{tikzcd}$ is an $ \Lambda V$-semi-free resolution of $(\mathbb{Q},0)$.
This defines a multiplication
\begin{equation*}\label{hommulti}
\mu_{\mathrm{A}}: \mathrm{A}\otimes \mathrm{A}\rightarrow \mathrm{A}
\end{equation*}
on $\mathrm{A} = Hom_{(\Lambda V,d)}(\Lambda V\otimes \Lambda sV, \Lambda V)$.
Therefore,  passing to cohomology we acquire a well-defined map of vector spaces:
\begin{align}
\mu_{\mathcal{A}}: \mathcal{E}xt_{(\Lambda V,d)}(\mathbb{Q},(\Lambda V,d)) \otimes_{\mathbb{Q}} \mathcal{E}xt_{(\Lambda V,d)}(\mathbb{Q},(\Lambda V,d)) & \longrightarrow \mathcal{E}xt_{(\Lambda V,d)}(\mathbb{Q},(\Lambda V,d))\\
\left[f\right]\otimes\left[g\right] \hspace*{3.4cm} & \longmapsto \left[f\cdot g\right]
\end{align}
which gives rise to a multiplication in $\mathcal{E}xt_{(\Lambda V,d)}(\mathbb{Q},(\Lambda V,d))$ and we can readily check that it is associative since the product in $\Lambda V$ is.
\\\\
\textbf{\textit{Uniqueness:}}

Let $m': (\Lambda V',d')\rightarrow A_{PL}(X)$ be another minimal Sullivan model of $X$. We know, by \cite[Proposition 12. 10]{FHT}, that  $(\Lambda V,d)$ and  $(\Lambda V',d')$ are isomorphic. The same goes clearly  for $Q$ and $Q'$. It results from the above commutative triangle that $\mu_A$ is independent of the choice of minimal model of $X$.
\\\\
\textbf{\textit{Unit element:}}

Let $\varepsilon :\Lambda V\rightarrow \mathbb{Q}$ be the augmentation. Recall that $P=\Lambda V\otimes \Lambda sV$ is a $(\Lambda V,d)$-semi-free resolution of $(\mathbb{Q},0)$.

We extend $\varepsilon$ to $\varepsilon '=\varepsilon \otimes \varepsilon_{\Lambda sV}: \Lambda V\otimes \Lambda sV\rightarrow \mathbb{Q}$, then we compose it with the injection $i:\mathbb{Q}\hookrightarrow\Lambda V$ and obtain $\tilde{\varepsilon}:\Lambda V\otimes \Lambda sV\rightarrow \Lambda V$ a representative of a class in $\mathcal{E}xt_{(\Lambda V,d)}(\mathbb{Q},(\Lambda V,d))$.
Now, for $f:\Lambda V\otimes \Lambda sV\rightarrow \Lambda V$, a representative of an arbitrary class in $\mathcal{E}xt_{(\Lambda V,d)}(\mathbb{Q},(\Lambda V,d))$, we have $$f\cdot\tilde{\varepsilon}:(\Lambda V\otimes \Lambda sV)\otimes_{\Lambda V}(\Lambda V\otimes \Lambda sV)=\Lambda V\otimes\Lambda sV\otimes \Lambda sV\rightarrow \Lambda V$$
and the map
$$
\begin{array}{cccc}   \theta =Id_{\Lambda V\otimes \Lambda sV}\otimes \varepsilon_{\Lambda sV}: & \Lambda V\otimes\Lambda sV\otimes \Lambda sV & \longrightarrow & \Lambda V\otimes\Lambda sV \\
  & 1\otimes sv\otimes 1 & \longmapsto & 1\otimes sv; \\
  & 1\otimes sv\otimes sw & \longmapsto & 0; \\
  & 1\otimes 1\otimes sv & \longmapsto & 0 \\
\end{array}
$$
makes the following diagram commutative:
$$\begin{tikzcd}[row sep=large,column sep=large]
\mathbb{Q} & \Lambda V\otimes \Lambda sV \arrow[l,"\simeq"']\arrow[d,"f"]\\
\Lambda V\otimes\Lambda sV\otimes \Lambda sV\arrow[u,"\simeq"]\arrow[ur,"\theta"]\arrow[r,"f\cdot\tilde{\varepsilon}"'] & \Lambda V.
\end{tikzcd}$$
Thus,  it defines a homotopy unit element for $\mathrm{A}=Hom_{(\Lambda V,d)}((P,d),(\Lambda V,d))$. Passing to cohomology,
 we get $[f]\cdot [\tilde{\varepsilon}]=[f]$ and  similarly $[\tilde{\varepsilon}]\cdot [f]=[f]$. Henceforth, the class $[\tilde{\varepsilon}]$ defines a unit element for $\mu_{\mathcal{A}}$.
\\\\
\textbf{\textit{Commutativity:}}

Let $\tau$ be the flip map $\tau :P\otimes_{\Lambda V} P\rightarrow P\otimes_{\Lambda V} P$; $x\otimes y\mapsto (-1)^{|x||y|}y\otimes x$.  The diagram
$$\begin{tikzcd}[row sep=large,column sep=large]
 & P\otimes_{\Lambda V} P\arrow[d,"(-1)^{|f||g|}g\cdot f"]\\
P\otimes_{\Lambda V} P \arrow[ur,"\tau"]\arrow[r,"f\cdot g"'] & \Lambda V
\end{tikzcd}$$
 is commutative.


clearly $\tau$ being a quasi-isomorphism,   $f\cdot g\sim g\cdot f$ and  $[f\cdot g]=(-1)^{|f||g|}[g\cdot f]$ so that, the multiplication on   $\mathrm{A}$ is homotopy commutative and that on $\mathcal{A}$    is indeed  commutative.

We  respectively conclude that  $\mathrm{A}$ is a homotopy commutative  differential graded algebra with unit and $\mathcal{E}xt_{(\Lambda V,d)}(\mathbb{Q},(\Lambda V,d))$ is a graded commutative $\mathbb{Q}$-algebra with unit. 

Finally, it is clear that  the following diagram, where $cev$ is the chain evaluation map of $(\Lambda V, d)$,  is commutative:
$$\begin{tikzcd}[row sep=large]
\mathrm{A} \otimes \mathrm{A} \arrow[r,"\mu_{\mathrm{A}}"]\arrow["cev\otimes cev", d] & \mathrm{A} \arrow[d, "cev"] \\
(\Lambda V,d)\otimes (\Lambda V,d) \arrow[r,"\mu_{\Lambda V}"] & (\Lambda V,d). 
\end{tikzcd}$$
Thus, passing to cohomology




$$\begin{tikzcd}[row sep=huge]
\mathcal{A} \otimes \mathcal{A} \arrow[r,"\mu_{\mathcal{A}}"]\arrow[d,"ev\otimes ev"'] & \mathcal{A} \arrow[d,"ev"] \\
H(\Lambda V,d)\otimes H(\Lambda V,d) \arrow[r] & H(\Lambda V,d),
\end{tikzcd}$$
we deduce that the evaluation map is a morphism of graded algebra.
\end{proof}

\section{$\mathcal{E}xt$-versions approximations and the main theorem}

Let $(A, d)$ be any commutative differential graded algebra model for a space $X$, $(\Lambda V, d)$ its minimal Sullivan model given by the quasi-isomorphism $\theta:(\Lambda V, d) \stackrel{\simeq}{\rightarrow}(A, d)$ \cite{FHT}.  Referring to \cite{JCar},
the cdga morphism
$$ \mu_{n}^{\theta}:=\left(\operatorname{Id}_{A}, \theta, \ldots, \theta\right):(A, d) \otimes(\Lambda V, d)^{\otimes n-1} \rightarrow(A, d)$$
is a special model, called an {\it $s$-model}, for the path fibration $\pi_{n}: X^{I} \rightarrow X^{n}$ which is indeed the  substitute of the $n$-fold diagonal  $\Delta^n_X: X\rightarrow X^n$.
This allows the following:
\begin{definition}
\begin{enumerate}
{\it\item[(a):] $\mathrm{TC}_{n}\left(X_{0}\right)$ is the least $m$ such that the projection
$$ \rho_{m}:\left(A \otimes(\Lambda V)^{\otimes n-1}, d\right) \rightarrow\left(\frac{A \otimes(\Lambda V)^{\otimes n-1}}{\left(\operatorname{ker} \mu_{n}^{\theta}\right)^{m+1}}, \overline{d} \right) $$
admits an algebra retraction.	
\item[(b):] $\mathrm{mTC}_{n}(\mathrm{X,\mathbb{Q}})$ is the least $m$ such that $\rho_{m}$ admits a retraction as $\left(A \otimes(\Lambda V)^{\otimes n-1}, d\right)$-module.
\item[(c):] $\mathrm{HTC}_{n}(X, \mathbb{Q})$ is the least $m$ such that $H\left(\rho_{m}\right)$ is injective.
\item[(d):] $nil \ker H^*(\Delta^n_{X}, \mathbb{Q})$ is the longest non trivial product of elements of $\ker H^*(\Delta^n_X), \mathbb{Q})$.}
\end{enumerate}
\end{definition}
It is also well known (cf. for instance \cite{JCar}) that:
\begin{equation}\label{inequa}
 nil \ker H^*(\Delta^n_X, \mathbb{Q})\leq   \mathrm{HTC}_{n}(X, \mathbb{Q})   \leq    \mathrm{mTC}_{n}(\mathrm{X}, \mathbb{Q}) \leq      \mathrm{TC}_{n}\left(X_{0}\right)\leq      \mathrm{TC}_{n}\left(X\right).
 \end{equation}

Now, if we take $\theta$ to be the identity of ${\Lambda V}$, $\mu_n^{\theta} = \mu_n^{Id_{\Lambda V}}$ becomes   the $n$-fold multiplication in $\Lambda V$ denoted by
\begin{equation}
\mu_n: (\Lambda V)^{\otimes n}\rightarrow \Lambda V.
\end{equation}
 Therefore, $ nil \ker H^*(\Delta^n_X, \mathbb{Q}) =  nil \ker H^*(\mu_n)$.

In a similar way, we put
\begin{equation}
\mu_{\mathrm{A},n} : \mathrm{A}^{\otimes n}\rightarrow \mathrm{A}\; \qquad \hbox{and} \;  \qquad \mu_{\mathcal{A},n} : \mathcal{A}^{\otimes n}\rightarrow \mathcal{A}
\end{equation}
 where $\mathrm{A} := Hom_{\Lambda V}((P,d),(\Lambda V,d))$ and $\mathcal{A}:=H(\mathrm{A})=\mathcal{E}xt_{(\Lambda V,d)}((P,d),(\Lambda V,d))$.
Then (cf. Definition \ref{sc-Ext}) the following definition  actually reads as the  $\mathcal{E}xt$-{\it version} of the previous one:
\begin{definition}\label{TC-Ext}
\begin{enumerate}
{\it\item[(a):] $\mathrm{TC}^{\mathcal{E}xt}_{n}\left(X, \mathbb{Q}\right)$ is the least $m$ such that the projection
$$\Gamma_{m}:\left(\mathrm{A}^{\otimes n}, d\right) \rightarrow\left(\frac{\mathrm{A}^{\otimes n}}{\left(\operatorname{ker} {{(\mu_{\mathrm{A},n})}}\right)^{m+1}}, \overline{d} \right)$$
admits a homotopy retraction.
\item[(b):] ${\mathrm{mTC}^{\mathcal{E}xt}}_{n}(\mathrm{X})$ is the least $m$ such that $\Gamma_{m}$ admits a homotopy retraction as $(\mathrm{A}^{\otimes n}, d)$-module.
\item[(c):] $\mathrm{HTC}^{\mathcal{E}xt}_{n}(X)$ is the least $m$ such that $H\left(\Gamma_{m}\right)$ is injective.
\item[(d):] $nil \ker{(\mu_{\mathcal{A},n}, \mathbb{Q})}$ is the longest non trivial product of elements of $\ker {(\mu_{\mathcal{A},n})}$.}
\end{enumerate}
\end{definition}
The same arguments used to establish the inequalities in (\ref{inequa}) allow the following:
\begin{equation}\label{ineq1}
nil \ker{(\mu_{\mathcal{A},n}, \mathbb{Q})}\leq   \mathrm{HTC}^{\mathcal{E}xt}_{n}(X, \mathbb{Q})   \leq    \mathrm{mTC}^{\mathcal{E}xt}_{n}({X}, \mathbb{Q}) \leq      \mathrm{TC}^{\mathcal{E}xt}_{n}\left(X_{0}\right).
\end{equation}

The projections $$\Gamma_{m}:\left(\mathrm{A}^{\otimes n}, d\right) \rightarrow\left(\frac{\mathrm{A}^{\otimes n}}{\left(\operatorname{ker} {{(\mu_{\mathrm{A},n})}}\right)^{m+1}}, \overline{d} \right) \;\;\; \hbox{and}\;\;\;
\rho_{m}:\left( (\Lambda V)^{\otimes n}, d\right) \rightarrow\left(\frac{(\Lambda V)^{\otimes n}}{\left(\operatorname{ker} \mu_{n}\right)^{m+1}}, \overline{d} \right) $$
induce
two short exact sequences linked by chain evaluation maps:
\begin{equation}\label{link}
\begin{tikzcd}
0 \arrow[r] & \left( \ker{{(\mu_{\mathrm{A},n})}} \right) ^{m+1} \arrow[r] \arrow[d] & \mathrm{A}^{\otimes n} \arrow[r,"\Gamma_m"] \arrow[d] & \dfrac{\mathrm{A}^{\otimes n}}{({\ker(\mu_{\mathrm{A},n})}) ^{m+1}} \arrow[r] \arrow[ "\theta", d] & 0\\
0 \arrow[r] & \left( ker\mu_{n} \right) ^{m+1} \arrow[r] & (\Lambda V)^{\otimes n} \arrow[r,"\rho_m"]  & \dfrac{ (\Lambda V)^{\otimes n} }{(\operatorname{ker}\mu_{n})^{m+1}} \arrow[r] & 0\\
\end{tikzcd}
\end{equation}

As a consequence, we have:
\begin{theo}
Let $X$ be  a $1$-connected finite type CW-complex.
 If $X$ is  a Gorenstein space over $\mathbb{Q}$ and $ev_{C^*(X, \mathbb{Q})}\not =0$,
then $HTC^{\mathcal{E}xt}_n(X, \mathbb{Q})\leq HTC_n(X, \mathbb{Q})
$  for any integer $n\geq 2$.
Furthermore, if $(\Lambda V, d)$ is a  Sullivan minimal model  of $X$, $[f]$ is the generating class of $\mathcal{A}$ and $m=HTC^{\mathcal{E}xt}_n(X, \mathbb{Q})$, then:
 $$HTC_n(X, \mathbb{Q}) = HTC^{\mathcal{E}xt}_n(X, \mathbb{Q})=:m\Leftrightarrow 
 f(1)^{\otimes n}\in (\ker{{\mu_{n}}})^{m}\backslash (\ker{{\mu_{n}}})^{m+1}.$$
\end{theo}
Before giving the proof of the theorem, let us first recall that if $X$ is a finite $n$-dimensional sub-complex of $\mathbb{R}^{n+k}$, $k>n+1$ and $M$ its regular neighborhood,  the homotopy fiber $F_X$ of the inclusion $\partial M\hookrightarrow M$ is called the  {\it Spivak fiber} for $X$ and  it is  a homotopy invariant of $X$. It is introduced in \cite{Sp}
 where it is shown that $X$ is a Poincaré complex if and only if $F_X$ is a homotopy sphere. 
 
The following corollary presents some essential classes of spaces satisfying the previous theorem:
\begin{cor}
Hypothesis of \ref{main1} and hence its conclusions are satisfied in the following cases:
\begin{enumerate}
\item[(a)]  $X$ is rationally elliptic,
\item[(b)]   $H_{>N}(X,\mathbb{Z})=0$, for some $N$,  and $H^*(X, \mathbb{Q})$ is a Poincaré duality algebra,
\item[(c)] $X$ is a  finite $1$-connected CW-complex and its Spivak fiber $F_X$ has finite dimensional cohomology.
\end{enumerate}
\end{cor}
 For the sake of completeness, we present below a sketch of the proof of the corollary:
\begin{enumerate}
\item[(a)] If $X$ is rationally elliptic, then it is Gorenstein and $ev_{C^*(X, \mathbb{Q})}\not =0$ thanks respectively  to \cite[Proposition 3.4]{FHT2} and \cite[Theorem A]{Mu}.
\item[(b)] Referring to \cite{FH1} (cf. also \cite[Theorem 3.6]{FHT2}), under the hypothesis,  $H^*(X, \mathbb{Q})$ is a Poincaré duality algebra  if and only if $X$ is a Gorenstein space over $\mathbb{Q}$. Thus, by \cite[Theorem 2.2]{FHT2} we have $\dim \pi_*(X)\otimes \mathbb{Q}<\infty$ and by \cite[Theorem A]{Mu} we obtain  $ev_{C^*(X, \mathbb{Q})}\not =0$.
\item[(c)] Here using \cite[Corollary 4.5]{FHT2} we have $H^*(X, \mathbb{Q})$ is a Poincaré duality algebra, hence it is a Gorenstein algebra. It results that   $X$ is Gorenstein over $\mathbb{Q}$ \cite[Proposition 3.2]{FHT2} and $ev_{C^*(X, \mathbb{Q})}\not =0$ as in the previous case.

\end{enumerate}
\begin{proof} {\bf (of Theorem 4.1)}:
Let $(\Lambda V,d)$ be a  Sullivan minimal model of $X$. Since $X$ is Gorenstein,  $\mathcal{A}\cong \mathbb{Q}\Omega$ where $\Omega$ is the generating class  represented by a cocycle $f\in \mathrm{A}^N$ of degree $N=fd(X)$, the formal dimension of $X$ (cf. \S 5, \cite{FHT2}). Therefore, the diagram (\ref{link}) induces in cohomology the following one 
\begin{equation*}
\begin{tikzcd}[column sep=small]
0 \arrow[r]\arrow[d] & H^{nN}(\ker{{(\mu_{\mathrm{A},n})}}^{m+1}) \arrow[r] \arrow[d] & (\mathcal{A}^N)^{\otimes n} \arrow[r,"H^{nN} (\Gamma_m)"] \arrow[d,"ev_{(\Lambda V,d)}^{\otimes n}"] & H^{nN}(\dfrac{\mathrm{A}^{\otimes n}}{{\ker(\mu_{\mathrm{A},n})}^{m+1}})  \arrow[r] \arrow["H^{nN}(\theta)", d] & 0\\
 H^{nN-1}(\dfrac{(\Lambda V)^{\otimes n}}{\left(ker\mu_{n} \right)^{m+1}}) \arrow[r] & H^{nN}(\left( ker\mu_{n} \right)^{m+1}) \arrow[r] & (H^{N}(\Lambda V))^{\otimes n} \arrow[r,"H^{nN}(\rho_m)"']  & H^{nN}(\dfrac{(\Lambda V)^{\otimes n}}{\left(ker\mu_{n} \right)^{m+1}})  \arrow[r] & 0\\
\end{tikzcd}
\end{equation*}
Now, since $ev_{(\Lambda V,d)}=ev_{C^*(X, \mathbb{Q})}\not =0$, this is also the case for   the horizontal arrow $ev_{(\Lambda V,d)}^{\otimes n}$. Thus, if $H^{nN}(\rho_m)$ is injective then   $H^{nN}(\Gamma_m)$ is also injective.
It results that: $HTC^{\mathcal{E}xt}_n(X, \mathbb{Q})\leq HTC_n(X, \mathbb{Q})$.

Next, let $m$ denote the smallest integer such that $H^{nN}(\Gamma_m)$ is injective or equivalently $f^{\otimes n}$ is a cocycle in $\mathrm{A}^{\otimes n}$ and $f^{\otimes n}\in \ker{{(\mu_{\mathrm{A},n})}}^{m+1}\backslash \ker{{(\mu_{\mathrm{A},n})}}^{m+1}$ (see Remarque below). Moreover, since $(\mathcal{A}^N)^{\otimes n}$ is one dimensional, $H^{nN}(\Gamma_m)$ is indeed a bijection. Hence,   $H^{nN}(\rho_{m})$ is injective if and only if 
$H^{nN}(\theta)$ is injective. But, $ev^{\otimes n}_{(\Lambda V,d)}$ being non-zero, the commutativity of the right diagram implies that this  is equivalent to  
 $f(1)^{\otimes n}\in (\ker{{\mu_{n}}})^{m}\backslash (\ker{{\mu_{n}}})^{m+1}.$ Notice that $\mu_{\mathrm{n}}(f(1)^{\otimes n})$ is a cocycle in $(\Lambda V)^{\otimes n}$ and $ev^{\otimes n}_{(\Lambda V,d)}[f^{\otimes n}]=[f(1)]^{\otimes n}\not =0$.
   It results that
$HTC_n(X, \mathbb{Q}) = HTC^{\mathcal{E}xt}_n(X, \mathbb{Q})=:m\Leftrightarrow 
f(1)^{\otimes n}\in (\ker{{\mu_{n}}})^{m}\backslash (\ker{{\mu_{n}}})^{m+1}.$
\end{proof}
\begin{Rq}
An equivalent definition of $HTC_n(X, \mathbb{Q})$ when $X$ is a Poincar\'e duality space reads as follows: It is the smallest integer $m\geq 0$ such that  some cocycle $\omega$ representing
 the fundamental class of $(\Lambda V,d)^{\otimes n}$, can be written as a product of $m$ elements of $ker(\mu_{n})$ (not necessarily cocycles).
Similarly, for any  Gorenstein space $X$,
$HTC^{\mathcal{E}xt}_n(X, \mathbb{Q})$ is the smallest integer $m$ such that  some cocycle representing the fundamental class of $\mathrm{A}^{\otimes n}$, namely $\Omega = [f]^{\otimes n}$ where $[f]$ designates the generating element of $\mathcal{A}^{N}$,
 can be written as a product of length $m$ of elements in $ker(\mu_{\mathrm A, n})$. {Therefore, in order to determine $HTC_n(X, \mathbb{Q})$ we may, using the precedent theorem,  calculate $m=HTC^{\mathcal{E}xt}_n(X, \mathbb{Q})$ which is quite simpler since $\mathcal{A}^*$ is one dimensional, and afterwards  deal with the obstruction to have the equality.}

Now, if $\dim V< \infty$ and $\dim H(\Lambda V,d)=\infty$ then by \cite[Theorem A]{Mu} we have $(\Lambda V,d)$  is a Gorenstein but not a Poincaré duality algebra. Moreover $ev_{(\Lambda V,d)} = 0$. 
{Hence, in this case, to compare the invariants $HTC_n^{\mathcal{E}xt}(X, \mathbb{Q})$ and $HTC_n(X, \mathbb{Q})$  we should
 determine them separately}.


\end{Rq}
\section{Use of the Adams-Hilton model.}


The Adams-Hilton model \cite{AH} of an arbitrary CW-complex $X$ over an arbitrary field $\mathbb{K}$ is a chain algebra quasi-morphism $\begin{tikzcd}[column sep=small] \theta_X :(TV,d)\arrow[r,"\simeq"] & C_{\ast}(\Omega X;\mathbb{K})\end{tikzcd}$ i.e.   $H_{\ast}(\theta_X)$ is an isomorphism of graded algebras. Here $V$ satisfies $H_{i-1}(V,d_1)\cong H_i(X;\mathbb{K})$ and $d_1: V\rightarrow V$ is the linear part of $d$.  $(TV, d)$ is called a {\it  free model} of $X$.


  Now, if $\mathbb{K}$ is a field with odd characteristic (thus containing $\frac{1}{2}$) and $X$ is a $q$-connected ($q\geq 1$) finite CW-complex such that $\dim X \leq q\cdot char(\mathbb{K})$ i.e. $X\in CW_q(\mathbb{K})$ (cf. Introduction) it has a minimal Sullivan model $(\Lambda W,d)$
  \cite[Theorem 7.1]{Ha}. Therefore,  using the isomorphism (\ref{natiso}) we have successively:
  \begin{equation}\label{iso1}
     \mathcal{E}xt_{C^{\ast}(X;\mathbb{K})}(\mathbb{K},C^{\ast}(X;\mathbb{K}))\stackrel{\cong}{\rightarrow} \mathcal{E}xt_{(\Lambda W,d)}(\mathbb{K},(\Lambda W,d)).
     \end{equation}
     and
    \begin{equation}\label{iso2}
         \mathcal{E}xt_{(TV,d)}(\mathbb{K},(TV,d))\stackrel{\cong}{\rightarrow} \mathcal{E}xt_{C_{\ast}(\Omega X;\mathbb{K})}(\mathbb{K},C_{\ast}(\Omega X;\mathbb{K})).
         \end{equation}
  These, combined with   the  isomorphism  of graded $\mathbb{K}$-vector spaces \cite[Theorem 2.1]{FHT2}:
   \begin{equation}\label{iso3}
   \mathcal{E}xt_{C^{\ast}(X;\mathbb{K})}(\mathbb{K},C^{\ast}(X;\mathbb{K}))\stackrel{\cong}{\rightarrow} \mathcal{E}xt_{C_{\ast}(\Omega X;\mathbb{K})}(\mathbb{K},C_{\ast}(\Omega X;\mathbb{K})).
   \end{equation}
   gives the isomorphism of  graded $\mathbb{K}$-vector spaces:
  \begin{equation}\label{AHan}
  \mathcal{E}xt_{(\Lambda W,d)}(\mathbb{K},(\Lambda W,d))\cong \mathcal{E}xt_{(TV,d)}(\mathbb{K},(TV,d)).
  \end{equation}
 Argument used in the rational case allows us to conclude that $\mathcal{E}xt_{(\Lambda W,d)}(\mathbb{K},(\Lambda W,d))$ has the structure of a graded commutative algebra with unit. The latter isomorphism serves to endow $\mathcal{E}xt_{(TV,d)}(\mathbb{K},(TV,d))$ with the same structure. It results the following
 \begin{prop}
 Let $\mathbb{K}$ be a field with odd characteristic and $X\in CW_q(\mathbb{K})$.
   Then, the graded vector spaces $\mathcal{E}xt_{C^{\ast}(X;\mathbb{K})}(\mathbb{K},C^{\ast}(X;\mathbb{K}))$ and $\mathcal{E}xt_{C_{\ast}(\Omega X;\mathbb{K})}(\mathbb{K},C_{\ast}(\Omega X;\mathbb{K}))$ have isomorphic  graded commutative algebra structures with unit. In particular, the Adams-Hilton model can be used to make this structure explicit.
 \end{prop}
 

 Recall that $\mathcal{E}xt_{(TV,d)}(\mathbb{K},(TV,d))$ is, as in the rational case, obtained  in terms of the acyclic closure of $\mathbb{K}$ of the form
$(TV\otimes (\mathbb{K}\oplus sV),\delta),$
where the differential $\delta$ satisfies  $\delta s+sd=id$, $d$ being the differential of $TV$. That is, for  any  element $z\otimes sv$  of $TV\otimes (\mathbb{K}\oplus sV)$, we have
$$\delta (z\otimes sv)=dz\otimes sv +(-1)^{|z|}zv\otimes 1-(-1)^{|z|}z\otimes sdv.$$
Notice that any element $f$ in $Hom^p_{(TV,d)}((TV\otimes (\mathbb{K}\oplus sV),\delta),(TV,d))$ is entirely determined by its image of $1\otimes (\mathbb{K}\oplus sV)$ since $TV\otimes (\mathbb{K}\oplus sV)$ is a left $(TV,d)$-module acting on the first factor. Thus we have
\begin{align*}
(D(g))(1\otimes sv) & =d\circ f(1\otimes sv)-(-1)^p f\circ\delta(1\otimes sv)\\
                    & =d f(1\otimes sv)-(-1)^{p(|v|+1)} vf(1)+(-1)^pf(1\otimes sdv).
\end{align*}
Therefore,
\begin{enumerate}
\item[(a) ]
An element $g$ in $Hom^{p-1}_{(TV,d)}((TV\otimes (\mathbb{K}\oplus sV),\delta),(TV,d))$ is in $Im(D)$ if and only if $g=D(f)$ for some $f$ in $Hom^p_{(TV,d)}((TV\otimes (\mathbb{K}\oplus sV),\delta),(TV,d))$, i.e.
$$g(1\otimes sv)=d f(1\otimes sv)-(-1)^{p(|v|+1)} vf(1)+(-1)^pf(1\otimes sdv).$$
Consequently:
\begin{equation}\label{Im}
g\in Im(D) \Leftrightarrow g(1\otimes sv) = d f(1\otimes sv)-(-1)^{p(|v|+1)} vf(1)+(-1)^pf(1\otimes sdv)\text{ for some }f.
\end{equation}

\item[(b)]
An element $f\in Hom^p_{(TV,d)}((TV\otimes (\mathbb{K}\oplus sV),\delta),(TV,d))$ is in $Ker(D)$ if and only if $D(f)=0$, that is, $d f(1\otimes sv)=(-1)^{p(|v|+1)} vf(1)-(-1)^pf(1\otimes sdv)$. Consequently:
\begin{equation}\label{Ker}
f\in Ker(D)\Leftrightarrow d f(1\otimes sv)=(-1)^{p(|v|+1)} vf(1)-(-1)^pf(1\otimes sdv).
\end{equation}
\end{enumerate}

Now, since $\operatorname{deg}(d)=-1$, $\mathcal{A}_*=\left(  Hom_{(TV,d)}((TV\otimes (\mathbb{K}\oplus sV),\delta),(TV,d) ),D\right)$ is  a $DGA_*$  in the sense of \cite{FHT2}. Using the standard convention $\mathcal{A}^{-q}=\mathcal{A}_{q},\; \text{for all } q\in \mathbb{Z}$, we obtain a  $DGA^*$ whose cohomology at $-p$ is the $\mathbb{K}$-module:
\begin{align*}
\mathcal{E}xt^{-p}_{(TV,d)}(\mathbb{Q},(TV,d)) = H_p\left(  Hom_{(TV,d)}((TV\otimes (\mathbb{K}\oplus sV),\delta),(TV,d) ),D\right)
\end{align*}
More explicitly, if $f\in Hom_{(TV,d)}(TV\otimes (\mathbb{K}\oplus sV), \Lambda V)$ is a cycle of (homological) degree $p$, it define a cohomological class $[f]\in \mathcal{E}xt^{-p}_{(TV,d)}(\mathbb{Q},(TV,d))$ of degree $-p$.



\subsection{Case of a suspension.} Assume that $\mathbb{K} = \mathbb{Q}$ and 
 let $X$ be a simply connected space, and $Y=\Sigma X$ its suspension.
 The morphism of graded modules $\sigma_{\ast} : H_{\ast}(X;\mathbb{Q})\longrightarrow H_{\ast}(\Omega \Sigma X;\mathbb{Q})$   
induced by the adjoint  $\sigma : X\rightarrow \Omega\Sigma X$ of $id_{\Sigma X}$ 
 extends to  a morphism of graded algebras
$T(\sigma_{\ast}):TH_{\ast}(X;\mathbb{Q})\longrightarrow H_{\ast}(\Omega \Sigma X;\mathbb{Q})$. In fact, this is an  isomorphism of graded algebras  
since $H_{\ast}(X;\mathbb{Q})$ is a free graded $\mathbb{Q}$-module. Therefore (\cite{BS})
$$\begin{tikzcd} (TH_{\ast}(X;\mathbb{Q}),0)\arrow[r,"\simeq"] & C_{\ast}(\Omega\Sigma X;\mathbb{Q}) \end{tikzcd}$$ is an Adams-Hilton model of $Y=\Sigma X$. It results that
 $$\mathcal{E}xt_{C_{\ast}(\Omega\Sigma X;\mathbb{Q})}(\mathbb{Q},C_{\ast}(\Omega\Sigma X;\mathbb{Q}))\cong \operatorname{Ext}_{TV}(\mathbb{Q},TV)$$  
 where  $V=H_{\ast}(X;\mathbb{Q})$.
Now of $(S,d)= (\Lambda W,d)$ is a Sullivan model of $Y=\Sigma X$, using (\ref{iso3}) we obtain a commutative diagram
$$\begin{tikzcd}[row sep=large]
\operatorname{Ext}_{TV}(\mathbb{Q},TV)\otimes_{\mathbb{Q}}\operatorname{Ext}_{TV}(\mathbb{Q},TV)\arrow[r,"\mu_{\mathcal{A}}"]\arrow[d,"\cong"'] & \operatorname{Ext}_{TV}(\mathbb{Q},TV)\arrow[d,"\cong"]\\
\mathcal{E}xt_{(S,d)}(\mathbb{Q},(S,d)) \otimes_{\mathbb{Q}} \mathcal{E}xt_{(S,d)}(\mathbb{Q},(S,d))\arrow[r,"\mu_{\mathcal{A}}"']
& \mathcal{E}xt_{(S,d)}(\mathbb{Q},(S,d)).
\end{tikzcd}$$

This permits the  use of Adams-Hilton models to explicitly describe the algebra structure on $\mathcal{A}$ 
since it restricts to ordinary $\operatorname{Ext}$ which loosen the calculations. 
Notice that $\Omega (\Sigma X)$ is weakly equivalent to the James space $J(X)$ and 
 referring to \cite[Example 7]{FHT} that there exists a minimal Sullivan model for $\Sigma X$ of the form $(\Lambda W, d)$ with quadratic differential  i.e. such that $d(W)\subset \Lambda ^2W$.
\subsection{When $X$ is a $2$-cell CW complex.} Let $\mathbb{K}$ any field containing $\frac{1}{2}$.
In this subsection, we showcase another use of the Adams-Hiton models to help picture the algebra structure on $\mathcal{A}$.
Let then  $X=S^q\cup_{\varphi} e^{q+1}$, $q\geq 2$, be the space where the cell $e^{q+1}$ is attached by a map of degree $r$. The Adams-Hilton model of $X$ has the form $(TV,d)$ where $V$ is a $\mathbb{K}$-vector space generated by $a$ and $a'$ with $deg(a)=q-1$,\quad $deg(a')=q$,\quad $da=0$ and\quad $da'=-ra$.


Let us go back to where we left off at the beginning of this section and apply, in this case, the obtained formulas (\ref{Im}) and (\ref{Ker}). This respectively  yields :

\begin{equation}\label{equaIm}
g\in Im(D) \Leftrightarrow \left \{ \begin{array}{l}
  g(1) =df(1),\\
  g(1\otimes sa)= df(1\otimes sa)-(-1)^{pq}af(1),\hspace*{1.4cm}(\text{ for some } f)\\
  g(1\otimes sa')= df(1\otimes sa')-(-1)^{p}rf(1\otimes sa)-(-1)^{p(q+1)}a'f(1),
  \end{array}
  \right.
\end{equation}
and
\begin{equation}\label{equaker}
f\in Ker(D)\Leftrightarrow \left \{ \begin{array}{l}
  df(1)=0,\\
  df(1\otimes sa)=(-1)^{pq}af(1),\\
  df(1\otimes sa')=(-1)^{p}rf(1\otimes sa)+(-1)^{p(q+1)}a'f(1).
  \end{array}
  \right.
\end{equation}

 Recall that to any pointed topological space $X$, it is associated in \cite{FHT2} an invariant called the {\it formal dimension} of $X$ (with respect to a field $\mathbb{K}$)  defined as follows:
$$fd(X, \mathbb{K}) = sup\{r\in \mathbb{Z}\; | \; [\mathcal{E}xt_{C^*(X; \mathbb{K})}^p(\mathbb{K},C^*(X; \mathbb{K}))]^r\not =0\}.$$
or $fd(X, \mathbb{K})=-\infty$ if such integer does not exist.
 In particular \cite[Proposition 5.1]{FHT2}, if $H^*(X; \mathbb{K})$ is finite dimensional,
$$fd(X, \mathbb{K}) = sup\{r\in \mathbb{Z}\; | \; {H^r(X; \mathbb{K})}\not =0\}.$$
Notice that, using cellular homology, we see that $H_*(X, \mathbb{Z})=H_0(X, \mathbb{Z})\oplus H_{q}(X, \mathbb{Z})\cong \mathbb{Z}\oplus \mathbb{Z}/r\mathbb{Z}$. We should then discuss two cases:
\begin{enumerate}
\item[(i)] If $char(\mathbb{K})=0$ or co-prime with $r$, we have $H^*(X, \mathbb{K})=H^0(X, \mathbb{K})\cong \mathbb{K}$. In this case, $H^*(X, \mathbb{K})$ has formal dimension $fd(X)=0$, thus, it is  a Poincaré duality space. Moreover, since it has finite dimensional cohomology, it is also a Gorenstein space \cite[Theorem 3.1]{FHT2}.
\item[(ii)] If $char(\mathbb{K})$ divides   $r$ then, $H^*(X, \mathbb{K})=H^0(X, \mathbb{K})\oplus H^{q}(X,\mathbb{K})\oplus H^{q+1}(X,\mathbb{K}) \cong \mathbb{K}\oplus \mathbb{K}\oplus\mathbb{K} $. Thus, since $q\geq 2$, $X$ is neither a  Poincaré duality space nor a  Gorenstein space \cite[Theorem 1]{FH1}. In this case, $fd(X)=q+1$, so that $\mathcal{E}xt^{k}_{(TV,d)}(\mathbb{K},(TV,d))=0, \forall k > q+1$.
\end{enumerate}\vspace*{0.1cm}
{\bf Example:}  In this example, we specify the case where $q=2$, i.e. $X=S^2\cup_{\varphi} e^3$. Thus $V=\mathbb{K}a\oplus \mathbb{K}a'$  with $|a|=1$ and $|a'|=2$.   We give below an  explicit computation of it to illustrate the use of Adams-Hilton models.

\textbf{i.}  {\it Assume that $char(\mathbb{K})=0$ or co-prime with $r$} (we specialize in the case where $\mathbb{K}=\mathbb{Q}$).

Let $f$ be a cycle of degree $0$, we have $df(1)=0$, then $f(1)$ is necessarily a scalar $f(1)=\gamma$. The second equation in (\ref{equaker}) implies that $df(sa)=\gamma a$, therefore $f(sa)=- \frac{\gamma}{r}a'+\gamma ' a^2$. The last equation in (\ref{equaker}) gives, after a simple simplification, $df(sa')=r\gamma 'a^2$, then $f(sa')= -\gamma_1'a'\cdot a + \gamma_2' a\cdot a' + \gamma'' a^3$,
$$Ker(D)=\dfrac{\mathbb{Q}\oplus \mathbb{Q}a' \oplus \mathbb{Q}a^2\oplus \mathbb{Q}a'\cdot a\oplus\mathbb{Q}a\cdot a'\oplus\mathbb{Q}a^3}{<x_1=-rx_2; x_3=-x_4+x_5>}.$$
Now let $g$ be an arbitrary element of degree $1$
$$\left\lbrace
\begin{array}{l}
g(1)=\alpha_1 a \\
g(sa)=\alpha_2 a'\cdot a + \alpha_3 a\cdot a' + \alpha_4 a^3 \\
g(sa')=\alpha_5 a'^2 + \alpha_6 a'\cdot a^2 + \alpha_7 a\cdot a'\cdot a + \alpha_8 a^2\cdot a' + \alpha_9 a^4,
\end{array}
\right.$$
by [\ref{equaIm}] we have
$$\left\lbrace
\begin{array}{l}
D(g)(1)=0 \\
D(g)(sa)=(-\alpha_1 -r \alpha_2 + r\alpha_3) a^2 \\
D(g)(sa')=(\alpha_1+r\alpha_2-r\alpha_5)a'\cdot a+ (r\alpha_3-r\alpha_5)a\cdot a'+(r\alpha_4-r\alpha_6+r\alpha_7-r\alpha_8)a^3,
\end{array}
\right.$$
therefore
$$Im(D)=\dfrac{\mathbb{Q}a^2\oplus \mathbb{Q}a'\cdot a\oplus\mathbb{Q}a\cdot a'\oplus\mathbb{Q}a^3}{<x_1=-x_2+x_3>}.$$
We consequently obtain $\mathcal{E}xt^0_{(TV,d)}(\mathbb{Q},(TV,d))=\mathbb{Q}$.

Applying the same process for $i\neq 0$, we recover the previously stated fact $\mathcal{E}xt^i_{(TV,d)}(\mathbb{Q},(TV,d))=0$.


\textbf{ii.} {\it  Assume  $char(\mathbb{K})$ divides $r$} so that $r=0$ (mod $char(\mathbb{K})$) and  $da=da'=0$.
Recall that in general, we have:
$$(D(f))(1\otimes sa)= df(1\otimes sa)-(-1)^{pq}af(1)=0,$$
and
$$(D(f))(1\otimes sa')= df(1\otimes sa')-(-1)^{p}rf(1\otimes sa)-(-1)^{p(q+1)}a'f(1)=0.$$
These become respectively in this case:
\begin{equation}\label{case (ii)}
(D(f))(1\otimes sa)= -(-1)^{pq}af(1)=0,
\text{ and }
(D(f))(1\otimes sa')= -(-1)^{p(q+1)}a'f(1)=0.
\end{equation}

Notice that in this case, for an element $f$ to be in $ker(D)$ it is necessarily that $f(1)=0$

Let $f$ be a cycle of degree $0$, then we have $f(1)=0$, consequently $df(sa)=df(sa')=0$, which implies $f(sa)=\gamma_1 a'+\gamma_2 a^2$ and $f(sa')= \gamma_3a'\cdot a + \gamma_4 a\cdot a' + \gamma_5 a^3$. Therefore
$$Ker(D)=\mathbb{K}a^2\oplus \mathbb{K}a' \oplus \mathbb{K}a'\cdot a\oplus\mathbb{K}a\cdot a'\oplus\mathbb{K}a^3.$$
Now let $g$ be an arbitrary element of degree $1$
$$\left\lbrace
\begin{array}{l}
g(1)=\alpha_1 a \\
g(sa)=\alpha_2 a'\cdot a + \alpha_3 a\cdot a' + \alpha_4 a^3 \\
g(sa')=\alpha_5 a'^2 + \alpha_6 a'\cdot a^2 + \alpha_7 a\cdot a'\cdot a + \alpha_8 a^2\cdot a' + \alpha_9 a^4,
\end{array}
\right.$$
hence 	
$$\left\lbrace
\begin{array}{l}
D(g)(1)=0 \\
D(g)(sa)=-ag(1)=-\alpha_1a^2 \\
D(g)(sa')=a'g(1)=\alpha_1a'\cdot a,
\end{array}
\right.$$
therefore
$$Im(D)=\dfrac{\mathbb{K}a^2\oplus \mathbb{K}a'\cdot a}{<x_1=-x_2>}\cong \mathbb{K}(a'\cdot a -a^2).$$
we obtain $\mathcal{E}xt^0_{(TV,d)}(\mathbb{K},(TV,d))\cong \mathbb{K}a^2\oplus \mathbb{K}a' \oplus\mathbb{K}a\cdot a'\oplus\mathbb{K}a^3\cong \mathbb{K}^{4}$.

An application of the same argument yields
\begin{align*}\mathcal{E}xt^{-1}_{(TV,d)}(\mathbb{K},(TV,d)) \cong \mathbb{K}a^3 \oplus  \mathbb{K}a\cdot a' \oplus \mathbb{K}a'\cdot a \oplus  \mathbb{K}a\cdot a'\cdot a\oplus \mathbb{K}a^2\cdot a'\oplus \mathbb{K}a^4 \cong \mathbb{K}^{6},
\end{align*}

$$
\begin{array}{lcl}
\mathcal{E}xt^{-2}_{(TV,d)}(\mathbb{K},(TV,d)) & \cong & \mathbb{K}a^4 \oplus \mathbb{K} a^2\cdot a' \oplus \mathbb{K}a\cdot a'\cdot a \oplus  \mathbb{K} a'\cdot a^2 \oplus \mathbb{K}a'^2  \oplus \\
& &  \mathbb{K}a\cdot a'^2 \oplus  \mathbb{K}a^3\cdot a'\oplus  \mathbb{K} a^2\cdot a'\cdot a \oplus \mathbb{K}a\cdot a'\cdot a^2\oplus \mathbb{K}a^5\\
& \cong & \mathbb{K}^{10}.
\end{array}
$$
From the previous cases, it is obvious that  $\mathcal{E}xt^{-i}_{(TV,d)}(\mathbb{K},(TV,d))\not =0,\; \forall i\geq 0$. Since $fd(X, \mathbb{K})=3$, we have
$\mathcal{E}xt^{i}_{(TV,d)}(\mathbb{K},(TV,d)) = 0, \; \forall i\geq 4$. It remains then to calculate the cases cases $i= -1, -2, -3$ which are given successively by applying the same process as follows\\
$\mathcal{E}xt^{1}_{(TV,d)}(\mathbb{K},(TV,d))\cong \mathbb{K}a'\oplus \mathbb{K}a^2$,
$\mathcal{E}xt^{2}_{(TV,d)}(\mathbb{K},(TV,d))\cong \mathbb{K}\oplus \mathbb{K}a$ and $\mathcal{E}xt^{3}_{(TV,d)}(\mathbb{K},(TV,d))\cong \mathbb{K}$.\\

Now notice that for $q>2$, Adams-Hilton model for $X=S^q\cup_{\varphi} e^{q+1}$ has generators $a$ and $a'$ of degrees respectively $q-1$ and $q$, thus the degree of $a'$ does not double that of $a$, so the previous example is somewhat a special case. However the computational process still holds, and we have, for the case \textbf{i}, $\mathcal{E}xt^{0}_{(TV,d)}(\mathbb{K},(TV,d))=\mathbb{K}$ and $\mathcal{E}xt^{i}_{(TV,d)}(\mathbb{K},(TV,d))=0$ for $i\neq 0$. Whereas for the case \textbf{ii}, computation process holds but the results differ, since we might have $\mathcal{E}xt^{-i}_{(TV,d)}(\mathbb{K},(TV,d))=0$ for finitely many $i\geq -(q+1)$ (e.g. for $q=7$ we have $\mathcal{E}xt^{7}_{(TV,d)}(\mathbb{K},(TV,d)) =0$), on the other hand, we always have $\mathcal{E}xt^{-i}_{(TV,d)}(\mathbb{K},(TV,d)) =0,\; \forall i<-(q+1)$ since $fd(X)=q+1$.

\end{document}